\documentclass[twoside,reqno]{amsart}
\usepackage{amsmath,amssymb,amsthm,amsfonts}
\usepackage{hyperref}

\usepackage{color}

\newtheorem{lemma}{Lemma}[section]
\newtheorem{theorem}{Theorem}[section]

\newtheorem{proposition}{Proposition}[section]
\newtheorem{remark}{Remark}[section]

\numberwithin{equation}{section}

\newcommand{\dis}{\displaystyle}

\newcommand{\R}{\mathbb{R}}

\renewcommand{\S}{\mathbb{S}}

\newcommand{\FP}{\mathbf{P}}

\newcommand{\FI}{\mathbf{I}}

\newcommand{\CM}{\mathcal{M}}

\newcommand{\na}{\nabla}

\newcommand{\al}{\alpha}
\newcommand{\be}{\beta}
\newcommand{\ga}{\gamma}
\newcommand{\om}{\omega}
\newcommand{\la}{\lambda}

\newcommand{\si}{\sigma}
\newcommand{\pa}{\partial}

\newcommand{\lag}{\langle}
\newcommand{\rag}{\rangle}

\newcommand{\eps}{\varepsilon}
\newcommand{\beg}{\begin{equation}}
\newcommand{\eng}{\end{equation}}
\newcommand{\bel}{\begin{equation}\label}
\newcommand{\enl}{\end{equation}}
\newcommand{\norm}[1]{\big\Vert#1\big\Vert}

\newcommand{\abs}[1]{\left\vert#1\right\vert}
\newcommand{\set}[1]{\left\{#1\right\}}

\newcommand{\inner}[1]{\left(#1\right)}

\newcommand{\comii}[1]{\left<#1\right>}

\newcommand{\reff}[1]{(\ref{#1})}

\begin{document}

\title[Hypocoercivity for the linear Boltzmann equation]{Hypocoercivity for the linear Boltzmann equation with confining forces}

\author[R.-J. Duan]{Renjun Duan}
\address[RJD]{Department of Mathematics, The Chinese University of Hong Kong,
Shatin, Hong Kong}
\email{rjduan@math.cuhk.edu.hk}

\author[W.-X. Li]{Wei-Xi Li}
\address[WXL]{School of Mathematics and Statistics, Wuhan University,
 430072 Wuhan, China}
\email{wei-xi.li@whu.edu.cn}



\begin{abstract}
This paper is concerned with the hypercoercivity property of solutions to the Cauchy problem on the linear Boltzmann equation with a confining potential force. We obtain the exponential time rate of solutions converging to the steady state under some conditions on both initial data and the potential function. Specifically, initial data is properly chosen such that the conservation laws of mass, total energy and possible partial angular momentums are satisfied for all nonnegative time, and a large class of potentials including some polynomials are allowed. The result also extends the case of parabolic forces considered in \cite{D-Hypo} to the non-parabolic general case here.
\end{abstract}

\maketitle
\thispagestyle{empty}

\tableofcontents

\section{Introduction}

Consider the following Cauchy problem on the linear Boltzmann equation with a stationary potential force
\begin{eqnarray}
 \pa_t F+\xi\cdot \na_x F-\na_x\phi\cdot \na_\xi F& =&Q(F,M),\label{eq}\\
 F(0,x,\xi) &=&F_0(x,\xi).\label{eq.0}
\end{eqnarray}
Here, $F=F(t,x,\xi)\geq 0$  stands for the density distribution function  of particles which have position
$x=(x_1,x_2,\cdots,x_n)\in\R^n$ and velocity
$\xi=(\xi_1,\xi_2,\cdots,\xi_n)\in\R^n$ at time $t\geq 0$, and the
initial data at time $t=0$ is given by $F_0=F_0(x,\xi)$. The integer $n\geq 2$ denotes the space dimension. The potential $\phi=\phi(x)$ of the external force is confining in the sense of
\begin{equation}\label{phi.c1}
    \int_{\R^n}e^{-\phi(x)}\,dx=1,
\end{equation}
and $M$ is a normalized global Maxwellian in the form of
$$
     M=(2\pi)^{-n/2}e^{-|\xi|^2/2}.
$$
$Q$ is a bilinear symmetric operator given by
\begin{equation}\label{def.Q}
Q(F,G)=\int_{\R^n\times \S^{n-1}} |\xi-\xi_\ast|^\ga q(\theta)(F'G_\ast'+F_\ast'G'-FG_\ast-F_\ast G)\,d\om d\xi_\ast.
\end{equation}
Here the usual notions $F'=F(t,x,\xi')$, $F_\ast=F(t,x,\xi_\ast)$, $F_\ast'=F(t,x,\xi_\ast')$ were used and likewise for $G$. $(\xi,\xi_\ast)$ and $(\xi',\xi_\ast')$, denoting  velocities of two particles
before and after their collisions  respectively, satisfy
\begin{eqnarray*}
\xi'=\xi-[(\xi-\xi_\ast)\cdot \om]\om,\ \
\xi_\ast'=\xi_\ast+[(\xi-\xi_\ast)\cdot \om]\om,\ \ \om\in \S^{n-1},
\end{eqnarray*}
by the conservation of momentum and energy
\begin{equation*}
   \xi+\xi_\ast=\xi'+\xi_\ast',\quad |\xi|^2+|\xi_\ast|^2=|\xi'|^2+|\xi_\ast'|^2.
\end{equation*}
The function $|\xi-\xi_\ast|^\ga q(\theta)$ in \eqref{def.Q} is the cross-section
depending only on $|\xi-\xi_\ast|$ and $\cos\theta=(\xi-\xi_\ast)\cdot \om/|\xi-\xi_\ast|$, and it is supposed to satisfy $\ga\geq 0$ and  the Grad's angular cutoff assumption $0\leq q(\theta)\leq q_0 |\cos \theta|$ for a positive constant $q_0$; see \cite{CIP-Book}.

It is well known that $Q$ has $n+2$ collision invariants $1$, $\xi_1$, $\xi_2$, $\cdots$, $\xi_n$ and $|\xi|^2$. Then, the linear Boltzmann equation \eqref{eq} has the conservation laws of mass and total energy
\begin{eqnarray}
&&\frac{d}{dt}\iint_{\R^n\times \R^n} F(t,x,\xi)\,dxd\xi=0,\label{dt.c1}\\
&&\frac{d}{dt}\iint_{\R^n\times \R^n} \left(\frac{1}{2}|\xi|^2+\phi(x)\right)F(t,x,\xi)\,dx d\xi =0,\label{dt.c2}
\end{eqnarray}
for any $t\geq 0$. The balance laws of momentum and angular momentum  are given by
\begin{eqnarray*}
&&\frac{d}{dt}\iint_{\R^n\times \R^n} \xi F(t,x,\xi)\,dxd\xi=-\int_{\R^n}\na_x\phi(x)\left[\int_{\R^n} F(t,x,\xi)\,d\xi\right]\,dx,\\
&&\frac{d}{dt}\iint_{\R^n\times\R^n}x\times \xi F(t,x,\xi)\,dxd\xi=-\int_{\R^n} x\times \na_x\phi(x)\left[\int_{\R^n} F(t,x,\xi)\,d\xi\right]\,dx.
\end{eqnarray*}
Here we used that for $i\neq j$,
\begin{equation*}
    \na_x \{(x\times \xi)_{ij}\}\cdot \xi=0,\quad \na_\xi\{ (x\times \xi)_{ij}\}\cdot \na_x\phi(x)=(x\times \na_x\phi(x))_{ij},
\end{equation*}
where for two vectors $A$ and $B$ the tensor product $A\times B$ is defined by $(A\times B)_{ij}=A_iB_j-A_jB_i$ for $i\neq j$ and zero otherwise. Since $\phi(x)$ is confining, it cannot be independent of some coordinate variable  of $(x_1,\cdots,x_n)$ so that the conservation of momentum generally does not hold. However the conservation of angular momentum is possible. For that, given $\phi(x)$, we denote an index set $S_\phi$ by
\begin{equation*}
S_\phi=\{(i,j); 1\leq i\neq j\leq n,\ x_i\pa_j\phi(x)-x_j\pa_i\phi(x)=0,\ \forall\,x\in \R^n\}.
\end{equation*}
Here and in the sequel, for brevity, $\pa_i$ denotes the spatial derivative $\pa_{x_i}$. It is now straightforward to check that \eqref{eq} also has the following conservation laws of angular momentum
\begin{equation}\label{dt.cij}
    \frac{d}{dt}\iint_{\R^n\times \R^n} (x\times \xi)_{ij} F(t,x,\xi)\,dxd\xi=0,\quad \forall\,(i,j)\in S_\phi,
\end{equation}
for any $t\geq 0$.

It is easy to see that the local Maxwellian
\begin{equation*}
    \CM=Me^{-\phi(x)}=(2\pi)^{-n/2}e^{-\left(|\xi|^2/2+\phi(x)\right)}
\end{equation*}
is a time-independent steady solution to the equation \eqref{eq}. We suppose that initial data $F_0(x,\xi)$ in \eqref{eq.0} has the same mass, angular momentum and total energy with $\CM$, namely
\begin{eqnarray}
&&\iint_{\R^n\times \R^n} [F_0(x,\xi)-\CM]\,dxd\xi=0,\label{con.F1}\\
&&\iint_{\R^n\times \R^n} (x\times \xi)_{ij} [F_0(x,\xi)-\CM]\,dxd\xi=0,\ (i,j)\in S_\phi,\label{con.F2}\\
&&\iint_{\R^n\times \R^n} \left(\frac{1}{2}|\xi|^2+\phi(x)\right)[F_0(x,\xi)-\CM]\,dx d\xi=0.\label{con.F3}
\end{eqnarray}
Due to the conservation laws \eqref{dt.c1}, \eqref{dt.cij} and \eqref{dt.c2}, the equations  above still hold true for the solution $F(t,x,\xi)$ to the Cauchy problem \eqref{eq}-\eqref{eq.0} at all positive time $t> 0$. Therefore, it could be expected that the solution $F(t,x,\xi)$  tends to the steady state $\CM$ exponentially in time as in the periodic domain \cite{Ukai-1974}. In general it is not the case for the whole space. However, we shall see that it can be achieved under some additional conditions on the potential function $\phi(x)$. Notice that the collision operator $Q$ is degenerate along $n+2$ number of directions. Even though this, the interplay between $Q$ and the linear transport operator $\pa_t+\xi\cdot \na_x-\na_x\phi\cdot \na_\xi$ involved in a confining force can indeed yield the  convergence of solutions to the steady state $\CM$ in an exponential rate. The property is now called hypercoercivity; see \cite{Vi}.

To the end, set
\begin{equation*}
    f(t,x,\xi)=\frac{F(t,x,\xi)-\CM}{\CM^{1/2}},\quad f_0(x,\xi)=\frac{F_0(x,\xi)-\CM}{\CM^{1/2}}.
\end{equation*}
Then the Cauchy problem  \eqref{eq}-\eqref{eq.0} is reformulated as
\begin{eqnarray}
 \pa_t f+\xi\cdot \na_x f-\na_x\phi\cdot \na_\xi f+Lf  = 0,  \quad
 f(0,x,\xi) =f_0(x,\xi), \label{req}
\end{eqnarray}
where $L$ is the usual linearized self-adjoint operator, defined by
\begin{equation*}
    Lf=-\frac{1}{M^{1/2}} Q(M^{1/2}f,M).
\end{equation*}
Recall that the null space  of $L$ is given by $\ker L=\{M^{1/2}, \xi_iM^{1/2},1\leq i\leq n, |\xi|^2M^{1/2}\}$, $L=-\nu+ K$ with $\nu=\nu(\xi)\sim (1+|\xi|)^\ga$ and $Ku=\int_{\R^3}K(\xi,\xi_\ast)u(\xi_\ast)\,d\xi_\ast$ for a real symmetric integral kernel $K(\xi,\xi_\ast)$, and moreover, $L$ satisfies the
coercivity  estimate:
 \[
\int_{\R^n} f Lf\,d\xi\geq \la_0\int_{\R^n}\nu(\xi) |\{\FI-\FP\} f|^2\,d\xi
\]
for a constant $\la_0>0$, where ${\bf I}$ is the identity operator and  ${\bf P} $ stands for the projection operator on $\ker L$
in $L_\xi^2$.

The main result of the paper concerning the hypercoercivity property of solutions to the Cauchy problem \eqref{req} is stated  as follows.

\begin{theorem}\label{thm.rate}
Assume that $\phi(x)$ is a confining potential in the sense of \eqref{phi.c1} and $F_0=\CM+\CM^{1/2}f_0$ satisfies the conservation laws \eqref{con.F1}, \eqref{con.F2} and \eqref{con.F3}. Additionally let the following assumptions on $\phi(x)$ hold:

\medskip

\noindent (i) For any $(i,j)\in S_\phi$, $\phi(x)$ is even in $(x_i,x_j)$, i.e.,
\begin{equation*}
\phi(\cdots,-x_i,\cdots,-x_j,\cdots)=\phi(\cdots,x_i,\cdots,x_j,\cdots),\quad x\in \R^n.
\end{equation*}

\noindent (ii) The functions $1$, $\{x_i\}_{1\leq i\leq n}$, $\{\pa_i\phi(x)\}_{1\leq i\leq n}$ and $\{ x_i\pa_j\phi(x)-x_j\pa_i\phi(x); 1\leq i<j\leq n, (i,j)\in I\setminus S_\phi\}$, where $I=\{(i,j);1\leq i\neq j\leq n\}$, are independent in the sense that if there are constants $b_0$, $\{b_{1i}\}_{1\leq i\leq n }$, $\{B_{1i}\}_{1\leq i\leq n}$ and $\{B_{2ij}; 1\leq i<j\leq n, (i,j)\in I\setminus S_\phi\}$ such that
\begin{equation*}
b_0+\sum_{i=1}^n b_{1i}x_i+\sum_{i=1}^n B_{1i}\pa_i\phi(x) +\sum_{1\leq i<j\leq n, (i,j)\in I\setminus S_\phi} B_{2ij}[x_i\pa_j\phi(x)-x_j\pa_i\phi(x)]=0
\end{equation*}
for any $x\in \R^n$ then all coefficients must be zero, i.e.,
\begin{equation*}
b_0=0;\ b_{1i}=0, 1 \leq i\leq n;\ B_{1i}=0, 1\leq i\leq n;\ B_{2ij}=0, 1\leq i<j\leq n, (i,j)\in I\setminus S_\phi.
\end{equation*}

\noindent (iii) If $\phi(x)$ is even, i.e., $\phi(-x)=\phi(x)$ for any
$x\in \R^n$, then the function
\begin{equation*}
\Lambda_\phi \Big(2\phi(x)+x\cdot \nabla_x \phi (x)\Big)+\frac{\abs{x}^2
    }{2}
\end{equation*}
with  $\Lambda_\phi$  a constant to be given in the later proof, is
not constant-valued; if $\phi$ is a general potential which may not be even, then
there exists a sequence $\set{x^\ell}_{\ell\geq 1}$,  with $\abs{x^\ell}\rightarrow
   +\infty $ as $\ell\rightarrow +\infty$,  such that
  \beg \label{hy1}
    \lim_{\ell\rightarrow +\infty} \frac{2\phi(x^\ell)+x^\ell\cdot
      \nabla_x \phi (x^\ell)+\abs{\nabla_x \phi (x^\ell)}}{\abs{x^\ell}^2}=0\ \ {\rm or}\ \ \lim_{\ell\rightarrow +\infty}\frac{
      \abs{x^\ell}^2+\abs{\nabla_x \phi (x^\ell)}}{2\phi(x^\ell)+x^\ell\cdot \nabla_x \phi (x^\ell)}=0.
   \eng

\medskip

\noindent
Then, for any  solution $f(t,x,\xi)\in L^2\inner{(0,1);
    L^2(\mathbb R^{2n})}$ to the Cauchy problem \eqref{req}, there are constants $\si>0$, $C$ such that
\begin{equation}\label{thm.r1}
    \|f(t)\|_{L^2_{x,\xi}}\leq C e^{-\si t}\|f_0\|_{L^2_{x,\xi}}
\end{equation}
for any $t\geq 0$.

\end{theorem}

A large class of potential functions satisfying conditions in Theorem \ref{thm.rate} can be allowed by just considering the polynomials in the form of
$$
P_{N}(x)=\sum_{|\al|\leq N} C_\al x^\al,
$$
where $\al=(\al_1,\cdots,\al_n)$ is a multi-index, $x^\al=x_1^{\al_1}\cdots x_n^{\al_n}$ and $C_\al$ are constants. Thus, the conditions on $\phi(x)=P_N(x)$ above are equivalent with ones by the proper choice of both $\al$ and $C_\al$. In what follows we shall give some simple examples. The first one is the radial function
$$
\phi_1(x)=\be\lag x \rag^{\al}+C_{\al\be},
$$
where $\lag x \rag:=(1+|x|^2)^{1/2}$, $\be>0$ and $0<\al\neq 2$ are constants, and $C_{\al\be}$ is a constant chosen such that the integral in \eqref{phi.c1} can be normalized to be unit. Notice that as long as the potential $\phi(x)$ is radial, $x\times \na_x\phi(x)\equiv 0$ and thus $S_\phi=\{(i,j); 1\leq i\neq j\leq n\}$, which implies that all angular momentums are conservative due to \eqref{dt.cij}.

\begin{remark}
The hypercoercivity property for $\phi=\phi_1(x)$ in the case when $\al=2$ was discussed in \cite{D-Hypo}, where the proof  is based on the approach of constructing the Lyapunov functional through  the analysis of the macroscopic system with the help of the macro-micro decomposition as well as the investigation of the Korn-type inequalities, and  more conservative quantities were assumed and the weighted $H^1_{x,\xi}$ norm instead of $L^2_{x,\xi}$ was used.
\end{remark}

The second example is the following non-radial but even function
$$
\phi_2(x)=\sum_{i=1}^n\be_i\lag x_i\rag^{\al_i}+C_{\al\be},
$$
where $0<\al_i\neq 2$ and $\be_i>0$ for $1\leq i\leq n$, and $C_{\al\be}$ depending on $\al=(\al_1,\cdots,\al_n)$ and $\be=(\be_1,\cdots,\be_n)$ is chosen as in $\phi_1(x)$. Notice that different from the case of $\phi_1(x)$, it is easy to see that $S_{\phi_2}$ is empty. For the non-radial and non-even potential, a typical example is given by
$$
\phi_3(x)=\sum_{i=1}^n\al_i x_i^{2(i+1)}+\sum_{i=1}^n\be_i x_i+C_{\al\be}
$$
for $\al_i>0$, $\be_i\neq 0$, $1\leq i\leq n$. One can also verify that $S_{\phi_3}$ is empty so that condition (i) in Theorem \ref{thm.rate} is automatically satisfied. Condition (ii) can be justified by comparing the coefficients of each monomial of different orders, and the second limit in \eqref{hy1} in condition (iii) also holds since the dominator grows strictly faster than the numerator at infinity. All details of verification are omitted for brevity.

\begin{remark}
From the later proof of Theorem \ref{thm.rate}, particularly basing on the proof of Lemma \ref{lem.z}, the conclusion \eqref{thm.r1} is still true if both conditions (ii) and (iii) in Theorem \ref{thm.rate} are combined together and replaced by the following new condition (ii)': If there are constants $b_0$, $\{b_{1i}\}_{1\leq i\leq n }$, $b_2$, $\{B_{1i}\}_{1\leq i\leq n}$, $\{B_{2ij}\}_{(i,j)\in \tilde{S}_\phi}$ and $B_3$ such that
\begin{multline*}
b_0+\sum_{i=1}^n b_{1i}x_i+b_2|x|^2+\sum_{i=1}^n B_{1i}\pa_i\phi(x) +\sum_{(i,j)\in \tilde{S}_\phi} B_{2ij}[x_i\pa_j\phi(x)-x_j\pa_i\phi(x)]\\
+B_3[2\phi(x)+x\cdot\na_x\phi(x)]=0
\end{multline*}
for any $x\in \R^n$ then all coefficients must be zero, i.e.,
\begin{equation*}
b_0=b_{1i}=b_2=B_{1i}=B_3=0, 1\leq i\leq n;\ B_{2ij}=0, (i,j)\in \tilde{S}_\phi.
\end{equation*}
Here, $\tilde{S}_\phi=\{1\leq i<j\leq n, (i,j)\in I\setminus S_\phi\}$.
\end{remark}

In what follows we mention some known work related to the result and approach of the paper. When there is no external force, the solution to the linearized Boltzmann equation \eqref{req} decays with the polynomial time rate due to \cite{Ukai-1974} using the spectral analysis. The polynomial rate was also obtained in \cite{Ka-BE13} through the compensation function method. Notice that \cite{D-Hypo} generalized this method to the case in the presence of a harmonic confining potential force. When the whole space is changed to the periodic domain, the approach of \cite{Ukai-1974} and \cite{Ka-BE13} still works under the additional assumptions on the conservation laws and then yields the exponential time-decay rate of solutions. See also \cite{DV,Guo-bd,Guo2,SG} for the study of the nonlinear Boltzmann equation. If a potential force is further added, the exponential rate was obtained in \cite{Duan-Torus} under some symmetry conditions on both the potential and initial data, and it was later extent in \cite{Kim} to the case without any symmetry restriction on the basis of the $L^\infty$-$L^2$ estimate developed in \cite{Guo-bd}.

On the other hand, the abstract theory in the context of hypercoercivity has been also studied in \cite{MN}, \cite{DMS} and \cite{Vi}, with general applications to the Fokker-Planck equation and the relaxation Boltzmann equation with the confining force in the whole space. Here we also mention the results about the exponential rate  through the hypoelliptic theory
developed in  \cite{HN} and \cite{H-AA,H-JFA}.

We conclude this section with explaining the idea in the proof of Theorem \ref{thm.rate}. In fact, in order prove Theorem \ref{thm.rate}, we need to obtain a criterion of zero solutions to the transport equation satisfying extra conservation laws; see Lemma \ref{lem.z}. Indeed, those conditions postulated in Theorem \ref{thm.rate} on the confining potential can assure the uniqueness of zero solutions. The proof of Lemma \ref{lem.z} is based on the analysis of the macroscopic equations \eqref{eq1.0}-\eqref{eq5} under the constraints \eqref{con1}-\eqref{con2}. Once the crucial Lemma \ref{lem.z} is established, it is a standard procedure as in \cite{Guo-bd} and \cite{Kim} to prove the refined coercivity estimate \eqref{coer} and hence derive the exponential time-decay rate in $L^2_{x,\xi}$.

\section{Criterion for zero solution}

This section is devoted to the proof of the following

\begin{lemma}\label{lem.z}
Let all conditions on $\phi(x)$ in Theorem \ref{thm.rate} hold. Then,
any function $z=z(t,x,\xi)\in L^2_t((0,1);L^2_{x,\xi})$ in the form of
\begin{equation}\label{lem.z.1}
    z(t,x,\xi)=\{a(t,x)+b(t,x)\cdot \xi+c(t,x)|\xi|^2\}\CM^{1/2}
\end{equation}
that satisfies in the sense of distribution the transport equation
\begin{equation}\label{lem.z.2}
    \pa_t z+\xi\cdot\na_x z-\na_x\phi\cdot \na_\xi z=0
\end{equation}
with the conservation laws
\begin{eqnarray}
&&\iint_{\R^n\times\R^n} z(t,x,\xi) \CM^{1/2}\,dxd\xi=0,\label{lem.z.c1}\\
&&\iint_{\R^n\times\R^n} (x\times \xi)_{ij} z(t,x,\xi) \CM^{1/2}\,dxd\xi=0,\ (i,j)\in S_\phi,\label{lem.z.c2}\\
&&\iint_{\R^n\times\R^n} \left(\frac{1}{2}|\xi|^2+\phi(x)\right)z(t,x,\xi) \CM^{1/2}\,dxd\xi=0,\label{lem.z.c3}
\end{eqnarray}
 for any $0\leq t\leq 1$ must be identical to zero, that is,
\begin{equation}\label{lem.z.de}
    z(t,x,\xi)\equiv 0,\quad \forall\, (t,x,\xi)\in [0,1]\times \R^n\times \R^n.
\end{equation}
\end{lemma}

We now start the proof of Lemma \ref{lem.z} by six steps.

\medskip

\noindent {\bf Step 1. }~
As in \cite{Guo2,Guo-IUMJ}, by plugging \eqref{lem.z.1} into \eqref{lem.z.2} and then collecting the coefficients of moments of different orders, $a(t,x)$, $b(t,x)$ and $c(t,x)$ satisfy a system of equations
\begin{eqnarray}
&&\dot{a}-\na_x\phi \cdot b=0,\label{eq1.0}\\
&&\dot{b}+\na_x a -2c \na_x\phi=0,\label{eq2}\\
&&\dot{c}+\pa_ib_i=0,\quad 1\leq i\leq n,\label{eq3}\\
&&\pa_jb_i+\pa_ib_j=0,\quad 1\leq i\neq j\leq n,\label{eq4}\\
&&\na_x c=0,\label{eq5}
\end{eqnarray}
for any $0\leq t\leq 1$ and $x\in \R^n$. Here and hereafter, for brevity, $\dot{a}$ denotes the time derivative of $a(t,x)$, and likewise for others. Moreover, putting \eqref{lem.z.1} into \eqref{lem.z.c1}, \eqref{lem.z.c2} and \eqref{lem.z.c3} gives
\begin{eqnarray}
&&\int_{\R^n} (a+A_1c)\,e^{-\phi}\,dx=0,\label{con1}\\
&&\int_{\R^n} (x\times b)_{ij}\,e^{-\phi}\,dx=0,\quad (i,j)\in S_\phi,\label{con.b}\\
&&\int_{\R^n}  \left(\frac{A_1}{2} a+\frac{A_2}{2} c\right)\,e^{-\phi}\,dx+\int (a+A_1 c)\,\phi e^{-\phi}\,dx=0,\label{con2}
\end{eqnarray}
for any $0\leq t\leq 1$, where the constants $A_1$ and $A_2$ depending only on the dimension $n$ are given by
\begin{eqnarray*}
&\dis A_1=\int_{\R^n} |\xi|^2M\,d\xi,\quad A_2=\int_{\R^n} |\xi|^4M\,d\xi.
\end{eqnarray*}

\medskip
\noindent {\bf Step 2.}~
First of all we determine $c(t,x)$. It is easy to see from \eqref{eq5} that
\begin{equation}\label{form.c}
 c(t,x)=c(t),
\end{equation}
i.e., $c$ is independent of $x$.  Next, we turn to $b(t,x)$. We shall use the same idea as in \cite{D-Hypo} for the proof of Korn-type inequality; see \cite[Step 4, Theorem 5.1]{D-Hypo}. Then, in terms of \eqref{eq3} and \eqref{eq4}, one has
\begin{equation}\label{form.b.0}
    b(t,x)=-\dot{c}(t)x+B_2(t)x +B_1(t),
\end{equation}
where $B_2(t)=\inner{B_{2ij}(t)}_{1\leq i,j\leq n}$ is  an $n\times
n$ skew-symmetric matrix and $B_1(t)=\inner{B_{11}(t),\cdots, B_{1n}(t)}$
is a vector,  both  depending only on $t$.  Plugging  the
expression \eqref{form.b.0} of $b(t,x)$ into \eqref{eq2}, it holds that
\begin{equation*}
    -\ddot{c}(t)x+\dot{B_2}(t) x+\dot{B_1}(t)+\na_x \{a(t,x)-2c(t)\phi(x)\}=0,
\end{equation*}
or equivalently
\begin{equation}\label{eq-p1}
    \dot{B_2}(t) x+\na_x
    \set{a(t,x)-2c(t)\phi(x)-\frac{1}{2} \ddot {c}(t)\abs x^2+\dot{B_1}(t)
      \cdot x }=0.
\end{equation}
This implies  that $\na_x\times \{\dot{B}_2(t) x\}=0$, i.e.,
\[
  \dot  B_{2ij}(t)-\dot B_{2ji}(t)=0,\quad 1\leq i\neq j\leq n.
\]
On the other hand, the skew-symmetry of the matrix $B_2(t)$ shows $
B_{2ij}(t)+ B_{2ji}(t)=0$ and thus
\[
\dot  B_{2ij}(t)+\dot B_{2ji}(t)=0.
\]
Then  $\dot  B_{2ij}(t)=0$ for any $t$ and all $i, j$. We have shown
\begin{equation}\label{11110302}
     \dot B_2(t)=0,
\end{equation}
that is, $B_2(t)=B_2$ is a skew-symmetric matrix with constant entries. As a
result, we may rewrite $a(t,x)$ and $b(t,x)$ as
\begin{equation}\label{form.a}
     a(t,x)=2c(t)\phi(x)+\frac{1}{2} \ddot {c}(t)\abs x^2-\dot{B_1}(t)
      \cdot x +d(t),
\end{equation}
and
\begin{equation}\label{+form.b}
  b(t,x)=-\dot c(t)x+B_2x+B_1(t),
\end{equation}
due to \reff{eq-p1} and \reff{11110302}, where $d(t)$ is a function
depending only on $t.$

\medskip
\noindent {\bf Step 3.}~
Plugging the expressions  \eqref{form.a} and \eqref{form.c} of $a(t,x)$ and $c(t,x)$ into \eqref{con1} and \eqref{con2},
one can deduce a second order linear system
\begin{eqnarray}
c(t)\la_1^1+\ddot{c}(t)\la_2^1+d(t)\la_3^1-\dot B_1(t)\cdot \la_4^1 &=&0,\label{sys1}\\
c(t)\la_1^2+\ddot{c}(t)\la_2^2+d(t)\la_3^2-\dot B_1(t)\cdot \la_4^2&=&0,\label{sys2}
\end{eqnarray}
where $\la_{j}^i$ $(i=1,2, j=1,2,3, 4)$ are constants depending only on $\phi(x)$, given by
\begin{equation*}
    \begin{array}{rlcrl}
      \dis  \la_1^1 &\dis= 2\int_{\R^n} (\frac{A_1}{2}+\phi)e^{-\phi}\,dx,&\qquad&
      \dis \la_1^2 &\dis= 2\int_{\R^n} (\phi^2+A_1\phi+\frac{A_2}{4})\,e^{-\phi}\,dx,\\[3mm]
       \dis  \la_2^1 &\dis= \int_{\R^n}\frac{1}{2}|x|^2\,e^{-\phi}\,dx,&\qquad&
      \dis \la_2^2 &\dis= \int_{\R^n}\frac{1}{2}|x|^2(\frac{A_1}{2}+\phi)\,e^{-\phi}\,dx,\\[3mm]
             \dis  \la_3^1 &\dis= \int_{\R^n}e^{-\phi}\,dx=1,&\qquad&
      \dis \la_3^2 &\dis= \int_{\R^n} (\frac{A_1}{2}+\phi)\,e^{-\phi}\,dx,\\[3mm]
             \dis  \la_4^1 &\dis= \int_{\R^n}xe^{-\phi}\,dx,&\qquad&
      \dis \la_4^2 &\dis= \int_{\R^n} (\frac{A_1}{2}+\phi) x \,e^{-\phi} \,dx.
    \end{array}
\end{equation*}
Moreover, note $\la_3^1=1>0$. In order to cancel $d(t)$, by multiplying \eqref{sys1} by $\la_3^2$ and then taking difference with \eqref{sys2}, it follows that
\begin{equation}\label{eq.cc}
    c(t)
    (\la_1^2-\la_1^1\la_3^2)+\ddot{c}(t)(\la_2^2-\la_2^1\la_3^2)-\dot
    B_1(t)\cdot \inner{\la_4^2-\la_4^1\la_3^2}=0,
\end{equation}
where one can compute
\begin{eqnarray}
  \la_1^2-\la_1^1\la_3^2&=& 2\left\{\int_{\R^n}\phi^2\,d\mu-\left[\int_{\R^n}\phi\,d\mu\right]^2+\frac{A_2-A_1^2}{4}\right\},\label{ceq.1}\\
\la_2^2-\la_2^1\la_3^2&=&\frac{1}{2}\left\{\int_{\R^n}|x|^2\phi\,d\mu -\int_{\R^n}|x|^2\,d\mu\,\int_{\R^n}\phi\,d\mu \right\},\label{ceq.2}\\
\la_4^2-\la_4^1\la_3^2&=&\int_{\R^n} x\phi\,d\mu-\int_{\R^n}x\,d\mu \int_{\R^n}\phi\,d\mu.\label{ceq.3}
\end{eqnarray}

Here and in the sequel, for brevity, we have used the notation
$
    d\mu=e^{-\phi(x)}\,dx.
$

\begin{lemma}\label{lem.sys}
It holds that
\begin{equation}\label{lem.sys.1}
    \la_1^2-\la_1^1\la_3^2>0,
\end{equation}
and
\begin{eqnarray}
\la_2^2-\la_2^1\la_3^2
 &=&\frac{1}{4}\iint_{\R^n\times \R^n}[\phi(x)-\phi(y)](|x|^2-|y|^2) e^{-\phi(x)-\phi(y)}\,dxdy,\label{lem.sys.2}\\
 \la_4^2-\la_4^1\la_3^2&=& \frac{1}{2}\iint_{\R^n\times \R^n} (\phi(x)-\phi(y))(x-y) e^{-\phi(x)-\phi(y)}\,dxdy.\label{lem.sys.3}
\end{eqnarray}
\end{lemma}

\begin{proof}
We verify \eqref{lem.sys.1} only. \eqref{lem.sys.2} and \eqref{lem.sys.3} can be proved from \eqref{ceq.2} and \eqref{ceq.3} in a similar way. In fact, it is straightforward to compute
\begin{eqnarray*}
&&\int_{\R^n}\phi^2\,d\mu-\left[\int_{\R^n}\phi\,d\mu\right]^2\\
&&=\iint_{\R^n\times \R^n} \frac{\phi^2(x)+\phi^2(y)}{2} e^{-\phi(x)-\phi(y)}\,dxdy-\iint_{\R^n\times \R^n} \phi(x)\phi(y)\,e^{-\phi(x)-\phi(y)}\,dxdy\\
&&=\frac{1}{2}\iint_{\R^n\times \R^n} [\phi(x)-\phi(y)]^2\,e^{-\phi(x)-\phi(y)}\,dxdy,
\end{eqnarray*}
and
\begin{eqnarray*}
&&A_2-A_1^2\\
&&=\int_{\R^n}|\xi|^4 M\,d\xi-\left[\int_{\R^n} |\xi|^2 M\,d\xi\right]^2\\
&&=\iint_{\R^n\times \R^n} \frac{|\xi|^4+|\xi_\ast|^4}{2}MM_\ast\,d\xi d\xi_\ast -\iint_{\R^n\times \R^n} |\xi|^2|\xi_\ast|^2MM_\ast\,d\xi d\xi_\ast\\
&&=\frac{1}{2}\iint_{\R^n\times \R^n} [|\xi|^2-|\xi_\ast|^2]^2 MM_\ast\,d\xi d\xi_\ast.
\end{eqnarray*}
Therefore, \eqref{lem.sys.1} follows in terms of \eqref{ceq.1}. This completes the proof of Lemma \ref{lem.sys}.
\end{proof}

Since $\la_1^2-\la_1^1\la_3^2>0$,  from \eqref{eq.cc}, one can write
\begin{equation}\label{cb1}
       c(t)=\Lambda_\phi\,\ddot c(t)+V_\phi\cdot\dot B_1(t),
\end{equation}
where the constant $\Lambda_\phi$ and the constant vector $V_\phi$, both depending only on $\phi(x)$, are given by
\begin{equation*}
    \Lambda_\phi=-\frac{\la_2^2-\la_2^1\la_3^2}{\la_1^2-\la_1^1\la_3^2},\quad V_\phi=\frac{\la_4^2-\la_4^1\la_3^2}{\la_1^2-\la_1^1\la_3^2}.
\end{equation*}

\medskip
\noindent {\bf Step 4.}~
Plugging  the expressions \reff{form.a} and \reff{+form.b}  of $a(t,x)$ and
$b(t,x)$  into \eqref{eq1.0},
\begin{equation}\label{11110801}
  \dot c(t)\inner{2\phi(x)+x\cdot \nabla_x \phi (x)}+\frac{1
    }{2} c^{(3)}(t) \abs x^2=\nabla_x \phi(x)\cdot B_2 x+\nabla_x \phi(x)\cdot
  B_1(t)+\ddot B_1(t)\cdot x -\dot d(t).
\end{equation}
Throughout this step, we consider the simplified case when the potential function $\phi(x)$ is even. In such case,
$\la_4^1=\la_4^2=0$ and hence $V_\phi=0$. Then, \eqref{cb1} is reduced to
\begin{equation}\label{+eq.cc}
    c(t)=\Lambda_\phi\, \ddot{c}(t).
\end{equation}
It follows from \reff{11110801} that
\begin{equation}\label{+11110801}
   \underbrace{\dot c(t)\inner{2\phi(x)+x\cdot \nabla_x \phi (x)}+\frac{1
    }{2} c^{(3)}(t) \abs x^2-\nabla_x \phi(x)\cdot B_2 x+\dot
    d(t)}_{\textrm{even function with respect to}~x}=0,
\end{equation}
and
\begin{equation}\label{++11110801}
  \underbrace{\nabla_x \phi(x)\cdot
  B_1(t)+\ddot B_1(t)\cdot x}_{\textrm{odd function with respect to}~x}=0.
\end{equation}
Due to the condition (ii) in Theorem \ref{thm.rate}, it is direct to obtain from \eqref{++11110801},
\begin{equation}\label{B10}
  \ddot B_1(t)=   B_1(t)\equiv 0
\end{equation}
for all $0\leq t\leq 1$. Moreover, note that $B_2$ is a constant matrix, then by \reff{+11110801} we have, taking the derivative with respect
to $t$,
\[
 \ddot c(t)\inner{2\phi(x)+x\cdot \nabla_x \phi (x)}+\frac{1
    }{2} c^{(4)}(t) \abs x^2+\ddot
    d(t)=0,
\]
which along with \reff{+eq.cc} give
\[
 \inner{\Lambda_\phi\inner{2\phi(x)+x\cdot \nabla_x \phi (x)}+\frac{\abs x^2
    }{2} }c^{(4)}(t)+\ddot
    d(t)=0.
\]
By using the condition (iii) for the case when $\phi(x)$ is even in Theorem \ref{thm.rate}, we conclude $c^{(4)}(t)\equiv 0$.
As a result, in view of \reff{+eq.cc}, it follows immediately that
$c(t)\equiv 0$,  and thus $d(t)\equiv 0$ by \eqref{sys1}. Then, further recalling \eqref{form.a} and \eqref{+form.b} as well as \eqref{B10}, we arrive at
\begin{equation}\label{s4-ab}
    a(t,x)\equiv 0,\quad b(t,x)=B_2x.
\end{equation}

\medskip
\noindent {\bf Step 5.}~
In this step we  prove \eqref{s4-ab} for the general case when $\phi(x)$ may not be even. In such case we shall use the other condition in (iii) in Theorem \ref{thm.rate}.
Note that \reff{11110801} still holds true and $B_2$ is independent of $t$. Then, taking the first order derivative with respect to $t$ on both sides of
\reff{11110801}, we obtain
\begin{equation}\label{ddotc}
  \ddot c(t)\inner{2\phi(x)+x\cdot \nabla_x \phi (x)}+\frac{1
    }{2} c^{(4)}(t) \abs x^2= \nabla_x \phi(x)\cdot
  \dot B_1(t)+ B_1^{(3)}(t)\cdot x +\ddot d(t).
\end{equation}

We first claim that $\ddot c(t)\equiv 0$ and $\ddot d(t)\equiv 0$.  Indeed,
the first relation in \reff{hy1} implies $c^{(4)}(t)\equiv 0$, while
the second one gives   $\ddot c(t)\equiv 0$. Moreover if
$c^{(4)}(t)\equiv 0$ then  we have,  taking derivatives twice with respect to
$t$ on both sides of  \reff{ddotc},
\begin{equation}\label{+ddotc}
    B_1^{(3)}(t)\cdot \nabla_x \phi(x)
+ B_1^{(5)}(t)\cdot x + d^{(4)}(t)=0.
\end{equation}
From the condition (ii) in Theorem \ref{thm.rate}, \eqref{+ddotc} implies
\begin{equation*}
    B_1^{(3)}(t)=B_1^{(5)}(t)\equiv 0,\quad  d^{(4)}(t)\equiv 0.
\end{equation*}
 Thus it follows from \reff{cb1} that $\ddot
c(t)\equiv 0$ and thus $\ddot d(t)\equiv 0$ due to \reff{sys1}.

The equation \reff{ddotc}  is now reduced to
\begin{equation*}
  \dot B_1(t)\cdot \na_x\phi(x)
+B_1^{(3)}(t)\cdot x=0,\quad \forall\,x\in\R^n.
\end{equation*}
Again from the condition (ii) in Theorem \ref{thm.rate}, we conclude
\[
     \dot
  B_1(t)\equiv 0,
\]
which together with  \reff{cb1} and  $\ddot
c(t)\equiv 0$  give
$
     c(t)\equiv 0,
$ and thus  $ d(t)\equiv 0$  due to \reff{sys1}. Therefore, similar to obtaining \eqref{s4-ab}, one has
$$
 a(t,x)\equiv 0,\quad b(t,x)=B_2x+B_1,
$$
where $B_1$ is constant.

\medskip
\noindent{\bf Step 6.}~
We are now going to determine $B_1=0$, $B_2=0$ and hence $b(t,x)\equiv 0$. In fact, the equation \reff{11110801} in both cases given in previous two steps is reduced to
\begin{equation}\label{b2}
  \nabla_x \phi(x)\cdot B_2 x+\nabla_x \phi(x)\cdot B_1=0,\quad \forall\,x\in \R^n.
\end{equation}
Noticing that $B_2$ is skew-symmetric, one can write
\begin{eqnarray*}
  \na_x \phi(x)\cdot B_2 x &=& \sum_{ij} B_{2ij}x_j\pa_i\phi(x)=\left\{\sum_{i<j}+\sum_{i>j}\right\} B_{2ij}x_j\pa_i\phi(x)\\
  &=&\sum_{1\leq i<j\leq n}\{B_{2ij}x_j\pa_i\phi(x)+B_{2ji}x_i\pa_j\phi(x)\}\\
   &=&-\sum_{1\leq i<j\leq
     n}B_{2ij}\{x_i\pa_j\phi(x)-x_j\pa_i\phi(x)\}\\
&=&-\sum_{\stackrel{1\leq i<j\leq n}{(i,j)\in I\setminus S_\phi}}B_{2ij}\{x_i\pa_j\phi(x)-x_j\pa_i\phi(x)\}.
\end{eqnarray*}
Therefore, \eqref{b2} together with the condition (ii) in Theorem \ref{thm.rate} yields
\begin{equation*}
   B_1=(B_{11}, \cdots, B_{1n})=0,\quad B_{2ij}=0\quad \text{for}\ 1\leq i<j\leq n, (i,j)\in I\setminus S_\phi,
\end{equation*}
which again due to the skew-symmetry of $B_2$, gives
\begin{equation}\label{B2-1}
     B_{2ij}=0\quad \text{for}\ 1\leq i\neq j\leq n, (i,j)\in I\setminus S_\phi.
\end{equation}
For $(i,j)\in S_\phi$, by plugging $b(t,x)=B_2x$ into \eqref{con.b},
\begin{eqnarray}
 0&=& \int_{\R^n}  \{x\times B_2 x\}_{ij}\,e^{-\phi(x)}\,dx\nonumber\\
 &=&\int_{\R^n} \{x_i[ B_2 x]_j-x_j[ B_2x]_i\}\,e^{-\phi(x)}\,dx\nonumber\\
  &=&\int_{\R^n} \Big\{x_i
    \sum_{\ell=1}^nB_{2j\ell} x_\ell-x_j\sum_{\ell=1}^nB_{2i\ell} x_\ell\Big\}\,e^{-\phi(x)}\,dx\nonumber\\
   &=&\int_{\R^n} \left\{x_i^2 B_{2ji} -x_j^2B_{2ij} \right\}\,e^{-\phi(x)}\,dx+\int_{\R^n} \Big\{x_i
    \sum_{\ell\neq i,j}B_{2j\ell} x_\ell-x_j\sum_{\ell\neq i,j}B_{2i\ell} x_\ell\Big\}\,e^{-\phi(x)}\,dx\nonumber\\
   &=&B_{2ji} \int_{\R^n} (x_i^2+x_j^2) \,e^{-\phi(x)}\,dx,\label{lem.p2}
\end{eqnarray}
where due to the condition (i) in Theorem \ref{thm.rate}, we have used
\begin{eqnarray*}
&&\int_{\R^n} \Big\{x_i
    \sum_{\ell\neq i,j}B_{2j\ell} x_\ell-x_j\sum_{\ell\neq i,j}B_{2i\ell} x_\ell\Big\}\,e^{-\phi(x)}\,dx\\
    &&=\int_{\R^n} \Big\{(-x_i)
    \sum_{\ell\neq i,j}B_{2j\ell} x_\ell-(-x_j)\sum_{\ell\neq i,j}B_{2i\ell} x_\ell\Big\}\,e^{-\phi(\cdots,-x_i,\cdots,-x_j,\cdots)}\,dx\\
 &&=-\int_{\R^n} \Big\{x_i
    \sum_{\ell\neq i,j}B_{2j\ell} x_\ell-x_j\sum_{\ell\neq i,j}B_{2i\ell} x_\ell\Big\}\,e^{-\phi(\cdots,x_i,\cdots,x_j,\cdots)}\,dx\\
    &&=0.
\end{eqnarray*}
Then it follows from \eqref{lem.p2} that
\begin{equation}\label{B2-2}
 B_{2ji}=-B_{2ij}=0,\quad \forall\,(i,j)\in S_\phi,
\end{equation}
since
\begin{equation*}
    \int_{\R^n} (x_i^2+x_j^2) \,e^{-\phi(x)}\,dx>0.
\end{equation*}
As a result, combining \eqref{B2-1} and \eqref{B2-2} implies $B_{2ij}=0$ for all $1\leq i,j\leq n$. Hence $B_2=0$ and
\begin{equation*}
b(t,x)=B_2x\equiv 0.
\end{equation*}
Therefore,
\begin{equation*}
    z(t,x,\xi)=\{a(t,x)+b(t,x)\cdot \xi+c(t,x)|\xi|^2\}\CM^{1/2}\equiv 0.
\end{equation*}
This proves \eqref{lem.z.de} and hence completes the proof of Lemma \ref{lem.z}.\qed

\section{Exponential rate}

In this section we use  $\comii{\cdot,~\cdot}$  and
$\norm{\cdot}$ to denote respectively
the inner product and norm in $L^2(\mathbb R^{n}_x\times\mathbb
R_\xi^n)$, while use
$\norm{\cdot}_\nu$ to denote the weighted norm with respect to $\nu=\nu(\xi)$,
i.e.,
$
\norm{\cdot}_\nu^2=\comii{\nu(\xi)\cdot,~\cdot}.
$
To the end, for brevity, $L^2_\nu$ also stands for  the weighted $L^2(\mathbb R^{n}_x\times\mathbb
R_\xi^n)$ space with norm
$\norm{\cdot}_\nu$.

Before proving Theorem \ref{thm.rate}, we give the following result on the basis of Lemma \ref{lem.z} and the property of the linearized collision operator $L$.

\begin{proposition}\label{prp.co}
Suppose that the potential function $\phi(x)$ satisfies all the conditions stated in Theorem \ref{thm.rate}. Let
$f(t,x,\xi)\in L^2\inner{[0,1];~L^2_\nu}$ be a solution of the linear Boltzmann equation \reff{req} with initial data $F_0=\CM+\CM^{1/2}f_0$ satisfying the conservation laws \eqref{con.F1}, \eqref{con.F2} and \eqref{con.F3}. Then there exists a constant $C>0$ such that
\begin{equation}\label{coer}
  \int_{0}^1 \comii{L f(s), f(s)}ds \geq C\int_{0}^1 \norm{f(s)}_{\nu}^2\,ds.
\end{equation}
\end{proposition}

\begin{proof}
We use the contradiction argument as did in
\cite{Guo2} and \cite{Kim} with a little modifications in order to treat the
whole spatial space case here instead of torus.

\medskip
\noindent{\bf Step 1.}~
Assume that the inequality \reff{coer} is
not true. Then for any $k\geq 1$ there exists a sequence of solutions
$z_k(t,x,\xi)$ such that
\[
  \int_{0}^1 \comii{L z_k(s), z_k(s)}ds < k^{-1}\int_{0}^1 \norm{z_k(s)}_{\nu}^2\,ds.
\]
Dividing both sides by the factor $\int_{0}^1 \norm{z_k(s)}_{\nu}^2\,ds$,
we may assume without loss of generality that
\begin{equation}\label{uni1}
  \int_{0}^1 \norm{z_k(s)}_{\nu}^2\,ds=1
\end{equation}
and
\bel{contra}
  \int_{0}^1 \comii{L z_k(s), z_k(s)}ds\leq k^{-1}, \quad \forall\,k\geq 1.
\enl
Note that  $\{z_k(t,x,\xi)\}_{k\geq 1}$ are  solutions to
\reff{req}, satisfying the conservation laws
\begin{eqnarray*}
&&\iint_{\R^n\times\R^n} z_k(t,x,\xi) \CM^{1/2}\,dxd\xi=0, \\
&&\iint_{\R^n\times\R^n} (x\times \xi)_{ij} z_k(t,x,\xi) \CM^{1/2}\,dxd\xi=0,\ (i,j)\in S_\phi, \\
&&\iint_{\R^n\times\R^n} \left(\frac{1}{2}|\xi|^2+\phi(x)\right)z_k(t,x,\xi) \CM^{1/2}\,dxd\xi=0,
\end{eqnarray*}
for any $0\leq t\leq 1$. By the weak compactness  in the Hilbert space $L^2\inner{[0,1]; L^2_\nu}$, we conclude that there exists
$z(t,x,\xi)\in L^2\inner{[0,1]; L^2_\nu}$ such that, up to a subsequence,
\bel{weakc}
     z_k\rightharpoonup z ~\textrm{ weakly in}~~ L^2\inner{[0,1]; L^2_\nu},
\enl
and
\[
   \int_{0}^1 \norm{z(s)}_{\nu}^2\,ds\leq 1,
\]
due to \reff{uni1}.   The weak convergence along with the above conservation laws yields that
$z$ also satisfies the same conservation laws as given in \eqref{lem.z.c1}, \eqref{lem.z.c2} and \eqref{lem.z.c3}.

\medskip
\noindent{\bf Step 2.}~
For each $r>0$,  we denote by $B_r\subset \mathbb R_x^n$ the ball with
center $0$ and radius $r$. In this step we claim   that for any fixed
$r$,
\bel{strongc}
K z_k\rightarrow K z  ~~\textrm{strongly in}~~L^2([0,1]\times
B_r\times \mathbb R^n).
\enl
The proof is nearly the same as  given in \cite{Guo2} and
\cite{Kim}, with $B_r$ instead of the torus there.  We omit  the details for brevity.

\medskip
\noindent{\bf Step 3.}~
We prove
\bel{eq1}
1-\int_{0}^1 \comii{Kz(s), z(s)}ds=0.
\enl
To do so, firstly we have,
using  the strong convergence \reff{strongc} and weak
convergence \reff{weakc},
\begin{equation}\label{conv}
   \int_{0}^1 \int_{B_r}\int_{\mathbb R^n} Kz(s,x,\xi) z(t,x,\xi)
   \,d\xi \,dx \,ds=\lim_{k\rightarrow +\infty} \int_{0}^1 \int_{B_r}\int_{\mathbb R^n} Kz_k(s,x,\xi) z_k(t,x,\xi)
   \,d\xi \,dx \,ds.
\end{equation}
Recall \eqref{uni1}. Then, for any  $\eps>0$,  we can find a constant $r_0$ depending only on $\eps$ but
independent of $k$, such  that
\[
     \forall~r\geq r_0,\quad \int_{0}^1 \int_{B_{r}}\int_{\mathbb R^n}
    \nu(\xi) \abs{z_k(s,x,\xi)}^2 \,d\xi \,dx\,ds  >1-\eps.
\]
This, together with \reff{contra} and the relation
\begin{eqnarray*}
   \int_{0}^1 \int_{B_r}\int_{\mathbb R^n} Kz_k(s,x,\xi) z_k(t,x,\xi)
   \,d\xi \,dx \,ds &=&\int_{0}^1 \int_{B_{r}}\int_{\mathbb R^n}
    \nu(\xi) \abs{z_k(s,x,\xi)}^2 \,d\xi \,dx\,ds\\
 &&-\int_{0}^1 \int_{B_r}\int_{\mathbb R^n} Lz_k(s,x,\xi) z_k(t,x,\xi)
   \,d\xi \,dx \,ds,
\end{eqnarray*}
implies
\begin{eqnarray*}
   \forall~r\geq r_0,\quad \int_{0}^1 \int_{B_r}\int_{\mathbb R^n} Kz_k(s,x,\xi) z_k(t,x,\xi)
   \,d\xi \,dx \,ds \geq  1-\eps-\frac{1}{k}.
\end{eqnarray*}
Note that  the number $r_0$
is independent of $k$. Then letting  $k\rightarrow+\infty$ we
have, by virtue of   \reff{conv},
\bel{below}
   \forall~r\geq r_0,\quad \int_{0}^1 \int_{B_r}\int_{\mathbb R^n} Kz(s,x,\xi) z(t,x,\xi)
   \,d\xi \,dx \,ds\geq  1-\eps.
\enl
Moreover the  identity
\[
 \int_{0}^1 \comii{Kz_k(s), z_k(s)}ds= \int_{0}^1 \norm{z_k(s)}_{\nu}^2
 ds-\int_{0}^1 \comii{Lz_k(s), z_k(s)}ds
\]
as well as \reff{uni1} and the positivity of $L$  give
\[
   \int_{0}^1 \int_{B_r}\int_{\mathbb R^n} Kz_k(s,x,\xi) z_k(t,x,\xi)
   \,d\xi \,dx \,ds \leq \int_{0}^1 \comii{Kz_k(s), z_k(s)}ds\leq 1,
\]
and thus
\begin{equation}\label{upp}
    \forall~r>0,\quad \int_{0}^1 \int_{B_r}\int_{\mathbb R^n} Kz(s,x,\xi) z(t,x,\xi)
   \,d\xi \,dx \,ds\leq 1
\end{equation}
due to \reff{conv}.   This along with \reff{below} shows that
\[
    \lim_{r\rightarrow+\infty} \int_{0}^1 \int_{B_r}\int_{\mathbb R^n} Kz(s,x,\xi) z(t,x,\xi)
   \,d\xi \,dx \,ds=1.
\]
As a result,  using the Lebesgue convergence theorem due to  \reff{upp},  we obtain
the desired estimate \reff{eq1}.

\medskip
\noindent{\bf Step 4.}~
We show
\begin{equation*}
   z(t,x,\xi)=\{a(t,x)+b(t,x)\cdot \xi+c(t,x)|\xi|^2\}\CM^{1/2}\ \text{and}\ z\not\equiv 0.
\end{equation*}
In fact, using the positivity of $L$ and the
fact that
\[
 \int_{0}^1 \comii{Lz(s), z(s)}ds = \int_{0}^1 \norm{z(s)}_{\nu}^2
 \,ds-\int_{0}^1 \comii{Kz(s), z(s)}ds \leq  1-\int_{0}^1 \comii{Kz(s), z(s)}ds=0
\]
due to \reff{eq1}, one has
\[
     \int_{0}^1 \comii{Lz(s), z(s)}ds=0,  \quad    \int_{0}^1 \norm{z(s)}_{\nu}^2
 \,ds=\int_{0}^1 \comii{Kz(s), z(s)}ds=1.
\]
Therefore, $z\not\equiv 0$, and
\[
     \{{\bf I}-{\bf P}\}z(t,x,\xi)=0
\]
for almost every $(t,x,\xi)\in[0,1]\times\mathbb R^{2n}$, by the coercivity estimate
\[
\int_{0}^1  \comii{Lz(s), z(s)} ds \geq C \int_{0}^1\|\{\FI-\FP\}z\|_\nu^2\,ds.
\]
Thus $z={\bf P}z=\{\tilde a(t,x)+\tilde b(t,x)\cdot \xi+\tilde c(t,x)|\xi|^2\}M^{1/2}=\{a(t,x)+b(t,x)\cdot \xi+c(t,x)|\xi|^2\}\CM^{1/2}$.

\medskip
\noindent{\bf Step 6.}~
We prove  that $z$ satisfies in the sense of distribution the transport equation
\[
 \pa_t z+\xi\cdot\na_x z-\na_x\phi\cdot \na_\xi z=0.
\]
For this purpose, for any $\psi\in
C_0^\infty\inner{[0,1]\times \mathbb R^n\times\mathbb R^n}$, we have
\begin{eqnarray*}
  -\iiint z_k \pa_t\psi \,d\xi\,dx\,dt -\iiint   z_k
  \xi\cdot\na_x \psi \,d\xi\,dx\,dt+\iiint  z_k
  \na_x\phi\cdot \na_\xi \psi \,d\xi\,dx\,dt&&\\
+\iiint \psi \nu(\xi) z_k   \,d\xi\,dx\,dt-\iiint \psi K z_k   \,d\xi\,dx\,dt&&  =0,
\end{eqnarray*}
since $z_k$ is a solution to \reff{req}.  On the other hand,  from the weak convergence
\reff{weakc} and strong convergence \reff{strongc},   it follows that, as $ k\rightarrow +\infty $,
\begin{multline*}
 -\iiint z_k \pa_t\psi \,d\xi\,dx\,dt -\iiint   z_k
  \xi\cdot\na_x \psi \,d\xi\,dx\,dt+\iiint  z_k
  \na_x\phi\cdot \na_\xi \psi \,d\xi\,dx\,dt\\
  +\iiint \psi \nu(\xi) z_k
  \,d\xi\,dx\,dt
 \rightarrow -\iiint z \pa_t\psi \,d\xi\,dx\,dt -\iiint   z
  \xi\cdot\na_x \psi \,d\xi\,dx\,dt\\
  +\iiint  z
  \na_x\phi\cdot \na_\xi \psi \,d\xi\,dx\,dt+\iiint \psi \nu(\xi) z
  \,d\xi\,dx\,dt,
\end{multline*}
and
\[
   \iiint \psi K z_k   \,d\xi\,dx\,dt\rightarrow \iiint \psi K z   \,d\xi\,dx\,dt.
\]
Combining the above relations, we conclude, for any  $\psi\in
C_0^\infty\inner{[0,1]\times \mathbb R^n\times\mathbb R^n}$,
\begin{eqnarray*}
  -\iiint z \pa_t\psi \,d\xi\,dx\,dt -\iiint   z
  \xi\cdot\na_x \psi \,d\xi\,dx\,dt+\iiint  z
  \na_x\phi\cdot \na_\xi \psi \,d\xi\,dx\,dt&&\\
+\iiint \psi \nu(\xi) z   \,d\xi\,dx\,dt-\iiint \psi K z   \,d\xi\,dx\,dt&&  =0.
\end{eqnarray*}
Moreover note  $z\in\ker L=\ker (\nu-K)$  in view of the conclusion of the
previous step.  Then the above
equation becomes
\begin{eqnarray*}
  \iiint z (t,x,\xi)\inner{-\pa_t-\xi\cdot\na_x +\na_x\phi\cdot \na_\xi} \psi(t,x,\xi)  \,d\xi\,dx\,dt =0,
\end{eqnarray*}
that is,  $z$ satisfies in the sense of distribution the equation
$\inner{\pa_t+\xi\cdot\na_x -\na_x\phi\cdot \na_\xi}z=0$. This together with the conservation laws \eqref{lem.z.c1}, \eqref{lem.z.c2} and \eqref{lem.z.c3} imply $z(t,x,\xi)\equiv 0$ by Lemma \ref{lem.z}, which is a contradiction with $z\not\equiv 0$. Then the proof
of Proposition \ref{prp.co} is complete.
\end{proof}

\begin{proof}[Proof of Theorem \ref{thm.rate}]
  Let $\lambda<1$ be a small number to be determined later, and let
  $f$ be a solution  to
  \reff{req} with initial data $F_0=\CM+\CM^{1/2}f_0$ satisfying the conservation laws \eqref{con.F1}, \eqref{con.F2} and \eqref{con.F3}.  Firstly we
  have the equation for $e^{\lambda t} f $:
\[
   \inner{ \pa_t +\xi\cdot\na_x -\na_x\phi\cdot \na_\xi+L} (e^{\lambda
     t} f )=\lambda e^{\lambda t} f.
\]
Then for any   $N\in\mathbb Z_+$,  one has
\begin{eqnarray*}
\int_0^N \comii{\inner{ \pa_s +\xi\cdot\na_x -\na_x\phi\cdot \na_\xi+L} (e^{\lambda
     s}f) ,~e^{\lambda s} f} ds= \lambda \int_0^N e^{2\lambda s}\norm{ f(s)}^2\,ds.
\end{eqnarray*}
This implies
\[
     e^{\lambda N} \norm{f(N)}^2+2\int_0^N  e^{2\lambda s}\comii{Lf(s), ~f(s)}\,ds=\norm{f(0)}^2+2 \lambda \int_0^N e^{2\lambda s}\norm{ f(s)}^2\,ds,
\]
since $\xi\cdot\na_x -\na_x\phi\cdot \na_\xi$ is skew-adjoint in
$L^2(\mathbb R_x^n\times\mathbb R^n_\xi)$.  We may rewrite the above
equation as
\begin{multline}\label{energy2}
 e^{\lambda N} \norm{f(N)}^2+2\sum_{k=0}^{N-1}\int_0^1  e^{2\lambda (s+k)}\comii{Lf_k(s),
    ~f_k(s)}\,ds\\
 =\norm{f(0)}^2+2\lambda\sum_{k=0}^{N-1}  \int_0^1 e^{2\lambda (s+k)}\norm{ f_k(s)}^2\,ds,
\end{multline}
where $f_k(s,x,\xi)\stackrel{\rm def}{=}f(s+k,x,\xi)$, $k=0,\cdots,
N-1$.  Note that $f_k$ satisfies the linear Boltzmann equation in the
interval $[0, 1]$. Then we use Proposition \ref{prp.co} to get
\begin{multline*}
    2\sum_{k=0}^{N-1}\int_0^1  e^{2\lambda (s+k)}\comii{Lf_k(s),
    ~f_k(s)}\,ds\geq 2 \sum_{k=0}^{N-1}e^{2\lambda k}\int_0^1  \comii{Lf_k(s),
    ~f_k(s)}\,ds\\
    \geq 2 C \sum_{k=0}^{N-1}e^{2\lambda k}\int_0^1  \norm{f_k(s)}^2\,ds,
\end{multline*}
where we have used the fact that $\nu_0\norm{\cdot}\leq \norm{\cdot}_\nu$
since  $\nu=\nu(\xi)$ has a positive lower bound $\nu_0$. On the other hand,
\[
   2\lambda\sum_{k=0}^{N-1}  \int_0^1 e^{2\lambda (s+k)}\norm{
     f_k(s)}^2\,ds\leq 2 e^{2\lambda }\lambda\sum_{k=0}^{N-1} e^{2\lambda k} \int_0^1 \norm{ f_k(s)}^2\,ds.
\]
As a result, we can choose $\lambda>0$ sufficiently small such that
\begin{eqnarray*}
    2\sum_{k=0}^{N-1}\int_0^1  e^{2\lambda (s+k)}\comii{Lf_k(s),
    ~f_k(s)}\,ds-2\lambda\sum_{k=0}^{N-1}  \int_0^1 e^{2\lambda (s+k)}\norm{
     f_k(s)}^2\,ds\\  \geq 2\inner{ C- e^{2\lambda }\lambda}
   \sum_{k=0}^{N-1} e^{2\lambda k} \int_0^1 \norm{ f_k(s)}^2\,ds
   \geq 0.
\end{eqnarray*}
This together with \reff{energy2} imply
\bel{uppb}
    e^{\lambda N} \norm{f(N)}^2\leq \norm{f(0)}^2.
\enl
Now, for any $t\geq 0$, we can find an integer $N_t$ such that $0\leq
N_t\leq t\leq N_t+1$. Using the relation
\begin{eqnarray*}
\int_{N_t}^t \comii{\inner{ \pa_s +\xi\cdot\na_x
       -\na_x\phi\cdot \na_\xi+L} f ,~ f} ds= 0,
\end{eqnarray*}
we get the $L^2$ energy estimate for the solution $f$ on the interval
$[N_t,t]$:
\[
   \norm{f(t)}^2+2\int_{N_t}^t  \comii{Lf(s), ~f(s)}\,ds=\norm{f(N_t)}^2.
\]
This, along with the positivity of $L$, yields
\[
   e^{\lambda t} \norm{f(t)}^2\leq e^{\lambda t} \norm{f(N_t)}^2=
   e^{\lambda (t-N_t)} e^{\lambda N_t}\norm{f(N_t)}^2\leq e\norm{f(0)}^2,
\]
where the last inequality follows from \reff{uppb} and the fact that
$e^{\lambda (t-N_t)}\leq e^{\lambda}\leq e $ since
$\lambda<1$. The desired estimate \eqref{thm.r1} follows by choosing $\si=\la/2$ and $C=\sqrt{e}$. This completes the proof of Theorem \ref{thm.rate}.
\end{proof}

\noindent{\bf Acknowledgements}: RJD was supported by a General
Research Fund from RGC of Hong Kong, and WXL was supported by the NSF of China (No. 11001207).


\end{document}